\documentclass[10pt]{amsart}
\usepackage{amscd,verbatim}
\usepackage{amssymb}
\usepackage[all]{xy}
\usepackage[inline]{enumitem}
\usepackage{hyphenat}
\usepackage[]{appendix}
\usepackage[colorlinks,linkcolor=blue,citecolor=blue,urlcolor=red]{hyperref}
\usepackage{cleveref}
\usepackage{xcolor}
\usepackage{mdframed}
\usepackage{bm}
\usepackage{amsmath}
\usepackage[normalem]{ulem}
\usepackage{mathtools}
\usepackage{tikz}
\usepackage{stmaryrd }
\usepackage[style=alphabetic,backref=false,doi=false,url=false,isbn=false,giveninits=false]{biblatex} 
\DeclareNameAlias{author}{last-first} 
\allowdisplaybreaks
\usepackage[a4paper,bindingoffset=0in,%
            left=1in,right=1in,top=1in,bottom=1in,%
            footskip=.25in]{geometry}

\makeatletter
\newcommand\subsubsubsection{\@startsection{paragraph}{4}{\z@}{-1ex\@plus 0ex \@minus 2ex}{-1ex \@plus 2ex\@minus 2ex}{\normalfont\normalsize\itshape}}
\newcommand\subsubsubsubsection{\@startsection{subparagraph}{5}{\z@}{-2.5ex\@plus -1ex \@minus -.25ex}{1.25ex \@plus .25ex}{\normalfont\normalsize\itseries}}
\makeatother

\makeatletter
\expandafter\def\csname r@tocindent4\endcsname{0pt}
\makeatother

\let\oldtocsection=\tocsection
\let\oldtocsubsection=\tocsubsection
\let\oldtocsubsubsection=\tocsubsubsection
\let\oldtocparagraph=\tocparagraph
\let\oldtocsubparagraph=\tocsubparagraph
\renewcommand{\tocsection}[2]{\hspace{0em}\oldtocsection{#1}{#2}}
\renewcommand{\tocsubsection}[2]{\hspace{1em}\oldtocsubsection{#1}{#2}}
\renewcommand{\tocsubsubsection}[2]{\hspace{2em}\oldtocsubsubsection{#1}{#2}}
\renewcommand{\tocparagraph}[2]{\hspace{3em}\oldtocparagraph{#1}{#2}}
\renewcommand{\tocsubparagraph}[2]{\hspace{4em}\oldtocsubparagraph{#1}{#2}}

\newcommand{\eq}[2]{\begin{equation}\label{#1}#2 \end{equation}}

\newtheorem{lemma}{Lemma}[section]

\newtheorem{theorem}[lemma]{Theorem}
\newtheorem{cor-intro}{Corollary}
\newtheorem{prop}[lemma]{Proposition}
\newtheorem{proposition}[lemma]{Proposition}

\newtheorem{corollary}[lemma]{Corollary}

\theoremstyle{definition}

\newtheorem{definition}[lemma]{Definition}

\newtheorem{para}[lemma]{}

\newtheorem{example}[lemma]{Example}
\newtheorem{remark}[lemma]{Remark}
\newtheorem{qn}[lemma]{Question}

\theoremstyle{remark}

\newtheorem*{claim*}{Claim}

\newcounter{zaehler}
\numberwithin{equation}{subsection}

\renewcommand{\phi}{\varphi}

\newcommand{\Z}{\mathbb{Z}}

\renewcommand{\P}{\mathbf{P}}
\newcommand{\A}{\mathbf{A}}

\newcommand{\rmv}{\mathrm{v}}

\newcommand{\cO}{\mathcal{O}}

\newcommand{\cM}{\mathcal{M}}

\newcommand{\Cone}{\operatorname{Cone}}

\newcommand{\ul}[1]{{\underline{#1}}}

\newcommand{\Tr}{\operatorname{Tr}}
\newcommand{\Nm}{\operatorname{Nm}}

\newcommand{\Spec}{\operatorname{Spec}}

\newcommand{\Proj}{\operatorname{Proj}}

\newcommand{\Zar}{{\operatorname{Zar}}}
\newcommand{\Nis}{{\operatorname{Nis}}}

\newcommand{\et}{{\operatorname{\acute{e}t}}}

\newcommand{\inj}{\hookrightarrow}

\newcommand{\res}{\operatorname{res}}

\newcommand{\id}{{\operatorname{id}}}

\newcommand{\CH}{{\operatorname{CH}}}

\newcommand{\ord}{{\rm ord}}

\newcommand*\circled[1]{\tikz[baseline=(char.base)]{
            \node[shape=circle,draw,inner sep=2pt] (char) {#1};}}

\newcommand{\leftrarrows}{\mathrel{\raise.75ex\hbox{\oalign{%
  $\scriptstyle\leftarrow$\cr
  \vrule width0pt height.5ex$\hfil\scriptstyle\relbar$\cr}}}}
\newcommand{\lrightarrows}{\mathrel{\raise.75ex\hbox{\oalign{%
  $\scriptstyle\relbar$\hfil\cr
  $\scriptstyle\vrule width0pt height.5ex\smash\rightarrow$\cr}}}}
\newcommand{\Rrelbar}{\mathrel{\raise.75ex\hbox{\oalign{%
  $\scriptstyle\relbar$\cr
  \vrule width0pt height.5ex$\scriptstyle\relbar$}}}}

\makeatletter
\def\leftrightarrowsfill@{\arrowfill@\leftrarrows\Rrelbar\lrightarrows}
\newcommand{\xleftrightarrows}[2][]{\ext@arrow 3399\leftrightarrowsfill@{#1}{#2}}
\makeatother

\newcommand{\pr}{{\mathrm{pr}}}

\renewcommand{\div}{\mathrm{div}}

\newcommand{\cyc}{{\mathrm{cyc}}}

\newcommand{\ra}{\rightarrow}

\renewcommand{\hat}{\widehat}
\renewcommand{\tilde}{\widetilde}
\newcommand{\xra}{\xrightarrow}

\newcommand{\hra}{\hookrightarrow}

\renewcommand{\bar}{\overline}

\newcommand{\bthm}{\begin{theorem}}
\newcommand{\ethm}{\end{theorem}}
\newcommand{\bcor}{\begin{corollary}}
\newcommand{\ecor}{\end{corollary}}
\newcommand{\bprop}{\begin{proposition}}
\newcommand{\eprop}{\end{proposition}}
\newcommand{\ble}{\begin{lemma}}
\newcommand{\ele}{\end{lemma}}
\newcommand{\bex}{\begin{example}}
\newcommand{\eex}{\end{example}}
\newcommand{\bexc}{\begin{exercise}}
\newcommand{\eexc}{\end{exercise}}
\newcommand{\brmk}{\begin{remark}}
\newcommand{\ermk}{\end{remark}}
\newcommand{\bdefn}{\begin{definition}}
\newcommand{\edefn}{\end{definition}}
\newcommand{\bpf}{\begin{proof}}
\newcommand{\epf}{\end{proof}}
\newcommand{\benu}{\begin{enumerate}}
\newcommand{\eenu}{\end{enumerate}}
\newcommand{\bit}{\begin{itemize}}
\newcommand{\eit}{\end{itemize}}
\newcommand{\bqn}{\begin{qn}}
\newcommand{\eqn}{\end{qn}}
\newcommand{\beq}{\begin{equation}}
\newcommand{\eeq}{\end{equation}}

\begin{document}

\title{Additive cycle complex and coherent duality}
\author{Fei Ren}
\maketitle
\begin{abstract}
Let $k$ be a field of positive characteristic $p$, and $X$ be a separated of finite type $k$-scheme of dimension $d$.
We construct a cycle map from the additive cycle complex to the residual complex of Serre-Grothendieck coherent duality theory.
This map is compatible with a cubical version of the map constructed in \cite{RenThesis} when $k$ is perfect.
As a corollary, we get injectivity statements for (additive) higher Chow groups as well as  for motivic cohomology (with modulus) with $\Z/p$ coefficients when $k$ is algebraically closed.
\end{abstract}
\tableofcontents

\section{Introduction}

This manuscript grows out of an attempt to extend the results in \cite{RenThesis} to the additive cycle complex settings. 
Let $X$ be a separated scheme of dimension $d$ of finite type over a perfect field $k$ of positive characteristic $p$.
Let $\pi:X\ra k$ be the structure map.
In \textit{loc. cit.}, we constructed an explicit chain map from Bloch's cycle complex $\Z^c_X$ of zero cycles to Kato's complex $K_{X,log}$, and showed that this map is a quasi-isomorphism in the \'etale topology if we mod the former complex by $p$. Kato's complex $K_{X,log}$ that we used over there, first appeared in \cite{KatoII}, comes from the Serre-Grothendieck duality theory for coherent sheaves. 
It is the $[-1]$-shifted mapping cone of $C-1$, where $C$ is the Cartier operator of an explicit complex of injectives $K_{X}$ representing $\pi^!k$.
In particular, we have a canonical chain map
$$K_{X,log}:=\Cone(K_X\xra{C-1} K_X)[-1]\ra K_{X}$$
which is at degree $-q$ given by $K_X^{-q}\oplus K_X^{-q-1}\ra K_X^{-q}, (a,b)\mapsto a$.
On the other hand, 
the (simplicial) cycle complex $\Z^c_X$ of Bloch is abstractly quasi-isomorphic to the cubical cycle complex $\Z_X^{\rm cube}$ by \cite[Theorem 4.7]{Levine_revisit} (See also \cite[Theorem 4.3]{Bloch_Notes}), and
for any integer $M\ge 2$, 
the closed immersion $i_{1}: X=X\times\{1\}\hra X\times \A^1$ induces a chain map
$$i_{1*}: \Z_X^{\rm cube} =\Z(d)_{X|0}[2d]\ra \Z(d+1)_{X\times \A^1 \mid M\cdot(X\times \{0\})}[2d+2]$$
by \cite[Lemma 2.7(1)]{BS19}.
Here $ \Z_{X\times \A^1 \mid M\cdot(X\times \{0\})}$ is the additive cycle complex introduced by Bloch-Esnault and Park, and we use the more general notation of the cycle complex with modulus for the modulus pair
$(X\times \A^1, M\cdot(X\times \{0\}))$.
(For a recollection of the construction of the cycle complex with modulus, see \Cref{Rec}.)
So a natural question is, whether there exists a chain map
$$\Z(d+1)_{X\times \A^1 \mid M\cdot(X\times \{0\})}[2d+2]\ra K_X$$
which is compatible with (the cubical version
\footnote{For this one just has to replace Zhong's map $\psi$ in \cite[\S6.1]{RenThesis} with a cubical version $\varphi'$ which is a sign adapted version of the map $\varphi$ in \cite[Lemma 3.1]{RS18}. See \Cref{phiprime}.} 
of) the chain map  defined in \cite{RenThesis}.
The main result of this article gives an affirmative answer to this question.
\begin{theorem}[{\Cref{thm}}, \Cref{Compatibility}]
Let $M\ge 2$ be an integer.
There is a chain map
$$\alpha:\Z(d+1)_{X\times \A^1 \mid M\cdot(X\times \{0\})}[2d+2]\ra K_X$$
between complexes of \'etale sheaves, such that  the following diagram commutes
$$\xymatrix{
\Z^{\rm cube}_X
\ar[r]^(0.3){i_{1*}}
\ar[d]_{\zeta\circ \varphi'}&
\Z(d+1)_{X\times \A^1 \mid M(X\times \{0\})}[2d+2]
\ar[d]^{\alpha}
\\
K_{X,log}
\ar[r]^{(a,b)\mapsto a}
&K_X.
}$$
\end{theorem}
As a corollary, we get the following injectivity statement.
\begin{corollary}[\Cref{injectivity}]
Suppose $k=\bar k$.
Then for any $q$, the following diagram 
$$\xymatrix{
\CH_0(X,q; \Z/p) \ar[r]^(0.33){i_{1*}}\ar@{=}[d] &
\CH_0(X\times \A^1| M(X\times {0}), q; \Z/p)\ar[d]
\\
H^{2d-q}_{\cM,\et}(X_\et, \Z/p(d)) \ar[r]^(0.33){i_{1*}} &
H^{2d+2-q}_{\cM,\et}({X\times \A^1| M(X\times {0})}, \Z/p(d+1)).
}$$
commutes, and the two horizontal maps  are injective.
\end{corollary}
\section{Recollection of cycle complexes with modulus}
\label{Rec}
Let $k$ be a field, $X$ be a separated of finite type $k$-scheme of dimension $d$, and $D$ be an effective Cartier divisor on $X$.
In this section we recall the definition of cycle complexes with modulus from \cite[Section 2]{BS19}.
The cycle complex with modulus is used to compute higher Chow groups with modulus, which is a generalization of the additive higher Chow groups defined by Bloch-Esnault \cite{BE1, BE2} and Park \cite{Park}.

\begin{para} 
Set $\P^1=\Proj k[Y_0,Y_1]$ and let $y=Y_1/Y_0$ be the standard coordinate function on $\P^1$.
We set 
\[\square=\P^1\setminus\{1\},\quad 
\bar \square=\P^1,\quad\text{and}\quad
\square ^q=(\P^1\setminus\{1\})^q,\quad
\bar \square ^q=(\P^1)^q, \text{ for }q>1
.\]
By convention we set $\square^0=\Spec k$.  Let $q_i: \bar \square^q\to \bar \square$ 
be the projection onto the $i$-th factor. We use the coordinate system $(y_1,\ldots, y_q)$ on $\bar \square^q$ with 
$y_i=y\circ q_i$. Let $F_i^q\subset \bar \square^q$ be the Cartier divisor defined by $\{y_i=1\}$ and 
put
$$F_q=\sum_{i=1}^q F^q_i.$$
A {\em face} of $\square^q$ is a subscheme $F$ defined by equations of the form 
\[y_{i_1}=\epsilon_1,\ldots, y_{i_r}=\epsilon_r,\quad 
                  r\in [1,q],\, (i_1,\ldots, i_r)\in [1,q]^r, \epsilon_{i_j}\in \{0,\infty\}.\]
We denote by $\imath_F: F\inj \square^q$ the closed immersion. For $\epsilon=0,\infty$  and $i\in [1,q]$, let 
\[\imath^q_{i,\epsilon}: \square^{q-1}\inj \square^q\]
be the inclusion of the face of codimension 1 given by $y_i=\epsilon$.
\end{para}

\begin{para}\label{defn:modulus-cycle}
Let $D$ be an effective Cartier divisor on $X$.
For $r,q\ge 0$ we denote by 
$$C_{d-r}(X|D,q)$$
(or $C^r(X|D,q)$, when $X$ is equidimensional) the set of all
integral closed subschemes 
$$Z\subset (X\setminus D)\times\square^q$$
of dimension $d+q-r$ which satisfy the following conditions:
\begin{enumerate}
\item $Z$ intersects  $(X\setminus D)\times F$ properly for all faces $F\subset \square^q$.
\item When $q=0$: The closure of $Z$ in $X$ does not meet $D$.
\item When $q\ge 1$: 
Denote by $\bar{Z}\subset X\times \bar \square^q$ the closure of $Z$ and 
             by $\nu_{\bar{Z}}: \tilde{Z}\to X\times \bar \square^q $ the composition 
                of the normalization $\tilde{Z}\to \bar{Z}$ 
           followed by the closed immersion $\bar{Z}\inj X\times \bar \square^q$.
          Then the following inequality between Cartier divisors  holds:
                   \eq{defn:modulus-cycle1}{\nu_{\bar{Z}}^*(D\times \bar \square^q)\le \nu_{\bar{Z}}^*(X\times F_q).}
\end{enumerate}
\end{para}

\begin{para}
Denote by $\ul{z}_{d-r}(X|D,q)$ (or $\ul{z}^r(X|D,q)$, when $X$ is equidimensional) the free abelian group on the set
$C_{d-r}(X|D,q)$.
By  \cite[Lemma 2.1]{BS19} there is a well-defined pullback map
$(\id_X\times\imath_F)^*: \ul{z}_{d-r}(X|D,q)\to \ul{z}_{d-r}(X|D,m)$ for any
$m$-dimensional face $\square^m\cong F\subset \square^q$. 
We obtain a cubical object of abelian groups (see e.g., \cite[\S1.1]{Le09}):
\[\ul{q}\mapsto \ul{z}_{d-r}(X|D,q) \quad (\ul{q}=\{0,\infty\}^q, q=0,1,2,3,\ldots).\]
For each $q$ we have the subgroup $\ul{z}_{d-r}(X|D,q)_{\rm degn}$ of degenerate cycles,
i.e. those cycles which come from $\ul{z}_{d-r}(X|D,q-1)$ via pullback along 
one of the $q$ projections $(X\setminus D)\times \square^{q}\to (X\setminus D)\times \square^{q-1}$.
We set 
\[z_{d-r}(X|D,q):=\frac{\ul{z}_{d-r}(X|D,q)}{\ul{z}_{d-r}(X|D,q)_{\rm degn}}.\]
(This is also denoted by $z^r(X|D,q)$ when $X$ is equidimensional.)
For any $q,r\ge 0$, $z_{d-r}(X|D,q)$ is a free abelian group.
The $q$-th boundary operator 
$\partial^\cyc: z_{d-r}(X|D,q)\to z_{d-r}(X|D,q-1)$
is given by 
\[\partial^\cyc= \sum_{i=1}^q (-1)^i (\partial^\infty_{i}-\partial^0_i),\]
where
$$\partial^\epsilon_i= (\id_X \times \imath^q_{i,\epsilon})^*: z_{d-r}(X|D,q)\to z_{d-r}(X|D,q-1)$$
is the pullback along the face $\{y_i=\epsilon\}$.
We get a complex $z_{d-r}(X|D,\bullet)$, which is the homological complex associated to the cubical object 
$\ul{q}\mapsto \ul{z}_{d-r}(X|D,q)$.
For any abelian group $A$, any integers $q,r\ge 0$, the association
$$U\mapsto z_{d-r}(U|U\times_X D,q)\otimes_{\Z}A$$
 is an \'etale sheaf on $X$ \cite[\S2.1.3]{BS19}.
\footnote{
The proof of this fact is completely analogous to the proof of \cite[Lemma 3.1]{Ge04} in the simplicial setting. One just has to note that $z_{d-r}(X|D,q)$ is not only a quotient but also a direct summand of $\ul z_{d-r}(X|D,q)$, and that the modulus condition behaves well with respect to \'etale pullbacks.
}
Set
$$A(r)^q_{X|D}(U):=  z_{d-r}(U|U\times_X D,2r-q)\otimes_{\Z} A.$$
For every $U\in X_\et$, $A(r)_{X|D}(U)$ is a complex of free $A$-modules. 
The association
$U\mapsto A(r)_{X|D}(U)$ is a complex of \'etale sheaves.
If $D=0$ we get back the cubical cycle complex with coefficients in $A$, and we omit ``$|D$'' from the notations in this case. 
\end{para}
\begin{para}
Let $A$ be an abelian group.
The higher Chow groups of $(X,D)$ with coefficients in $A$ are defined to be 
\[\CH_{d-r}(X|D,q;A)
:= H_q(z_{d-r}(X|D,\bullet)\otimes_{\Z} A)
=H^{2r-q}(A(r)_{X|D}(X)), \quad q,r\ge 0,\]
see \cite[Definition 2.5]{BS19}.  
(This is also denoted by $\CH^r(X|D,q;A)$ when $X$ is equidimensional.)
If $D=0$ we get back Bloch's classical definition of 
higher Chow groups by \cite[Theorem 4.3]{Bloch_Notes}, hence we can simply write  $\CH_{d-r}(X,q;A)$ instead of $\CH_{d-r}(X|0,q;A)$. 
For $\tau=\Zar, \Nis$ or $\et$, the $\tau$ motivic cohomology of $(X,D)$
with coefficients in $A$ are defined to be 
\[
H^i_{\cM,\tau}(X|D,A(r)):=
R^i\Gamma(X_\tau,A(r)_{X|D}), \quad r\ge 0.\]
see \cite[Definition 2.10]{BS19}.  
In general, the natural map
$$\CH_{d-r}(X|D,q;A)\ra H^{2r-q}_{\cM,\tau}(X|D,A(r))$$
for any given $r,q\ge 0$ is not an isomorphism, see the discussion in \cite[\S2.1.4]{BS19} and the counterexample for zero cycles with $\tau=\Nis$ in \cite[\S10]{GK}.
However, when $D=0$ and $\tau=\Zar$, it is an isomorphism for any $r,q\ge 0$.
This is the Zariski descent for higher Chow groups.
\end{para}

\section{A cubical version of Zhong's comparison theorem}
Let $k$ be a field and $X$ be a separated of finite type $k$-scheme of dimension $d$.

\begin{para}
Let $r\ge 0$ be an integer.
Denote by $C^M_{X,r}$  the Gersten complex of Milnor $K$-theory $K^M_{X,r}$. Namely,
$$C^M_{X,r}:=\left(
\bigoplus_{x\in X_{(d)}} \iota_{x*}K^M_{r}(k(x))\xra{\partial^M}
\bigoplus_{x\in X_{(d-1)}} \iota_{x*}K^M_{r-1}(k(x))\xra{\partial^M}
\dots\xra{\partial^M}
\bigoplus_{x\in X_{(0)}} \iota_{x*}K^M_{r-d}(k(x))
\right).
$$
Here $\iota_{x*}:x\hra X$ is the natural inclusion.
We regard $C^M_{X,r}$ as a complex of \'etale sheaves sitting in degrees $[r,d+r]$, namely,
$$(C^M_{X,r})^{q}=\bigoplus_{x\in X_{(d+r-q)}} \iota_{x*}K^M_{2r-q}(k(x)).$$
The  differentials $\partial^M$ are defined in a standard way so that $C^M_{X,r}$ is a cycle module in the sense of Rost, see \cite[(2.1.0)]{Rost}.

\end{para}

\begin{para} 
\label{notforcubical}
We fix some notations.
Let $Z\subset X\times  \square^q$ be a $a$-dimensional prime cycle. 
$W$ is the schematic image of $Z$ under the map 
$\bar \pr:X\times \bar \square^q\ra X$. 
$\bar Z$ is the closure of $Z$ in $X\times \bar \square^q$.
In the following diagram, the horizontal maps are normalizations, and the two unlabeled arrows are the natural open immersions. 
The maps $\pr: X\times  \square^q\ra X$, $\bar \pr: X\times \bar \square^q\ra X$ induce the maps $\pr:Z\ra W$ and $\bar \pr:\bar Z\ra W$ on the right. 
The maps $\tilde \pr:\tilde Z\ra \tilde W$ and $\tilde{\bar \pr}:\tilde{\bar Z}\ra \tilde W$ are given by the universal property of the normalization.
$$\xymatrix{
&
\tilde{\bar Z}
\ar[rr]^{\rho_{\bar Z}}\ar[ddl]^(.6){\tilde{\bar \pr}}
&&
\bar Z
\ar[ddl]^(.6){\bar \pr}
\\
\tilde Z
\ar[rr]^{\rho_Z} \ar@{^(->}[ur]\ar[d]_{\tilde \pr} 
&&
Z
\ar@{^(->}[ur]\ar[d]_{\pr}
\\
\tilde W
\ar[rr]^{\rho_W}
&&
W
}$$
We will use lowercase letters to represent dimension $a-1$ points of the corresponding scheme, e.g., $z\in Z_{(a-1)}, \tilde z\in \tilde Z_{(a-1)}, \tilde{\bar z}\in \tilde{\bar Z}_{(a-1)}, w\in W_{(a-1)}$, etc. ($w$ is not necessarily codimension 1 point in $W$!) The symbol $\eta$ with subscript will denote the generic point of the corresponding integral scheme, e.g., $\eta_W,\eta_Z$, etc.. $\rmv_{\tilde z}(-), \rmv_{\tilde{\bar z}}(-)$ etc., denotes the valuation.
\end{para} 

\begin{para} 
\label{phiprimer}
With our degree convention and notations from \Cref{notforcubical}, we adapt the map from \cite[\S3.1]{RS18} and define
\begin{align*}
\varphi'(r)^{2r-q}: 
\Gamma(X,\Z(r)_X^{2r-q})=z_{d-r}(X,q)
&\ra 
\Gamma(X,(C^M_{X,r})^{2r-q})
=\bigoplus_{x\in X_{(d-r+q)}} K^M_q(k(x))
\\
Z&\mapsto
\begin{cases}
\Nm_{k(\eta_{\bar Z})/k(\eta_W)}(\{y_1,\dots,y_q\}),&
\text{$\bar Z\xra{\bar \pr} W$ is gen. fin.}\\
0,&
\text{$\bar Z\xra{\bar \pr} W$ is not gen. fin.}
\end{cases}
\end{align*}
\end{para}
\begin{lemma}
For any $r\ge 0$,
$\varphi'(r):\Gamma(X,\Z(r)_{X})\ra \Gamma(X,C^M_{X,r})$ is a chain map.
\end{lemma}
\begin{proof}
Consider the diagram
$$\xymatrix{
z_{d-r}(X,q)
\ar[r]^(0.4){\varphi'(r)}
\ar[d]^{\partial^{\cyc}}
& 
\bigoplus_{x\in X_{(d-r+q)}} K^M_q(k(x))
\ar[d]^{\partial^M}
\\
z_{d-r}(X,q-1)
\ar[r]^(0.4){\varphi'(r)}
&
\bigoplus_{w\in X_{(d-r+q-1)}} K^M_{q-1}(k(w))
}$$
commutes.
Take any $(d-r+q)$-dimensional prime cycle $Z\subset X\times \square^q$.
Since our $\varphi'$ only differs from $\varphi$ by a sign, 
it suffices to check the commutativity over those points
$w\in X_{(d-r+q-1)}$ 
such that $\partial^M\circ\varphi'(Z)\neq 0$, 
and the commutativity of the rest of the cases follows from \cite[Lemma 3.1]{RS18}.
That is, it suffices to check over 
$w\in X_{(d-r+q-1)}\cap \bar\pr(Z)$
which is the projection of a point
$z\in Z_{(d-r+q-1)}$ lying in a face
$\{y_i(z)=\epsilon(z)\}$ with $i\in \{1,\dots,q\}$, $\epsilon(z)\in\{0,\infty\}$.
For a fixed $i$, one has $y_i=u(i,\tilde{\bar z})\cdot \pi_{\tilde{\bar z}}^{\rmv_{\tilde{\bar z}}(y_i)}$ for some $u(i,\tilde{\bar z})\in \cO_{\tilde{\bar Z},\tilde{\bar z}}^*$.
Denote by 
\eq{tamesymbol}{\partial_{\tilde{\bar Z},\tilde{\bar z}}:
K^M_q(k(\eta_{\tilde{\bar Z}}))\ra 
K^M_{q-1}(k({\tilde{\bar z}})),\qquad
\{ \pi_{\tilde{\bar z}},s_2,\dots,s_q\}\mapsto \{s_2,\dots,s_q\}
}
the tame symbol.
(Our sign convention for the tame symbol is the same as \cite{Totaro}, and differs from \cite{RS18} by a sign.)
On the one hand,
\begin{align*}
\partial^{M}_{w}\circ \varphi'(r)(Z)
&=
\partial^{M}_{w}\left(
\Nm_{k(\eta_{\bar Z})/k(\eta_W)}( \{y_1,...,y_q\})
\right)
\\
&\stackrel{(a)}{=}
\sum_{\tilde{\bar z}\in (\bar \pr \circ \rho_{\bar Z})^{-1}(w)}
\Nm_{k({\tilde{\bar z}})/k(w)}
\left(
\partial^{}_{\tilde{\bar Z},\tilde{\bar z}}
( \{y_1,...,y_q\})
\right)
\\
&=
\sum_{\tilde{\bar z}\in (\bar \pr \circ \rho_{\bar Z})^{-1}(w)\cap \tilde Z^{(1)}}
(-1)^{i(\tilde z)-1}
\rmv_{\tilde z}(y_{i(\tilde{\bar z})})
\Nm_{k({\tilde{\bar z}})/k(w)}
( \{y_1,...,\hat{y_{i(\tilde{\bar z})}},...,y_q\}).
\end{align*}
Here (a) comes from the definition of the differential $\partial^M$ in the complex $C^M_{X,r}$ (with our sign convention for the tame symbol \eqref{tamesymbol}), functoriality of the norm map and compatibility of the norm with the tame symbol (see e.g., \cite[R3b]{Rost}).
On the other hand,
\begin{align*}
\varphi'(r)_w\circ \partial^{\cyc}(Z)
&=
\varphi'(r)_w\left(
\sum_{i=1}^q (-1)^{i-1}
\left[
\sum_{z\in Z^{(1)}}
\ord_z(y_{i})\cdot \bar{\{z\}}^Z
\right]
\right)
\nonumber
\\
&=
\sum_{z\in \pr^{-1}(w)\cap Z^{(1)}}
(-1)^{i(z)-1} \ord_z(y_{i(z)}) 
\Nm_{k(z)/k(w)}
( \{y_1,...,\hat{y_{i(z)}},...,y_q\})
\nonumber
\\
&=
\sum_{z\in \pr^{-1}(w)\cap Z^{(1)}}
(-1)^{i(z)-1} \left(
\sum_{\tilde{z}\in \rho_{Z}^{-1}(z)}
\rmv_{\tilde z}(y_{i(z)})[k(\tilde z):k(z)]
\right)
\\
&\qquad\qquad\qquad\qquad\qquad
\cdot
\Nm_{k(z)/k(w)}
( \{y_1,...,\hat{y_{i(z)}},...,y_q\})
\nonumber
\\
&\stackrel{(b)}{=}
\sum_{z\in \pr^{-1}(w)\cap Z^{(1)}}
(-1)^{i(z)-1}
\sum_{\tilde{z}\in \rho_{Z}^{-1}(z)}
\rmv_{\tilde z}(y_{i(z)})
\Nm_{k(\tilde z)/k(w)}
(\{y_1,...,\hat{y_{i(z)}},...,y_q\})
\nonumber
\\
&=
\sum_{\tilde{z}\in (\pr\circ \rho_{Z})^{-1}(w)\cap \tilde Z^{(1)}}
(-1)^{i(\tilde z)-1}
\rmv_{\tilde z}(y_{i(\tilde z)})
\Nm_{k(\tilde z)/k(w)}
( \{y_1,...,\hat{y_{i(\tilde z)}},...,y_q\}).
\nonumber
\end{align*}
Here (b) follows from e.g., \cite[R2d]{Rost}.
Note that the index sets of these two last formulas are the same. 
We are hence done with the proof.
\end{proof}
The following proposition is a cubical version of \cite[Theorem 2.16]{Zhong}.
We warn the reader that here we are using the motivic theoretic symbols while \textit{loc. cit.} the author is using cycle complex conventional symbols.
In particular, our $\Z(r)_X$ corresponds to his $\Z^c_X(d-r)$ and our $C^M_{X,r}$ corresponds to his $C^M_X(d-r)$ (both up to signs).
\begin{prop}
\label{Zhongcubical}
Let $r\ge 0$,
$n\ge 1$ be two integers.
Then 
$$\phi'(r)/p^n:\Z(r)_{X}(X)/p^n\ra C^M_{X,r}(X)/p^n$$
is a quasi-isomorphism.
\end{prop}
\begin{proof}
The proof below follows completely from the outline of Zhong's proof, except for the point that we replace the simplicial map of Nesterenko–Suslin by the cubical map of Totaro.
We record the proof for the convenience of the reader.
Consider the niveau spectral sequence of higher Chow groups
$$E^1_{s.t}= \bigoplus_{x\in X_{(s)}} 
\CH_{d-r}(k(x),s+t-2d+2r;\Z/p^n)
\Rightarrow
\CH_{d-r}(X,s+t-2d+2r;\Z/p^n)
=H^{2d-s-t}(\Z(r)_X(X)/p^n) $$
By \cite[Theorem 1.1]{GL00}, $\CH_{d-r}(k(x),s+t-2d+2r;\Z/p^n)
=0$ for all $t\neq d-r$. Hence the spectral sequence degenerate at $E_1$, and 
we have canonical isomorphisms
$$H^{d-s+r}(\Z(r)_X(X)/p^n)\simeq H_s(E_{\bullet,d-r}^1)$$
for all $s$, given by the edge morphism of the spectral sequence. 
Instead of using the simplicial map of Nesterenko–Suslin \cite{NS}, we use the cubical map of Totaro \cite[\S3]{Totaro} to identify
$$
\CH_{d-r}(k(x),s-d+r;\Z/p^n)
\xra{\simeq}
K^M_{s-d+r}(k(x))/p^n,
\quad 
Z\mapsto \Nm_{k(Z)/k(x)}(\{y_1,\dots,y_{s-d+r}\}).
$$
Hence we have an isomorphism of complexes
$$E^1_{\bullet,d-r}\simeq C^M_{X,r}(X)/p^n,$$
which is precisely our $\phi'(r)/p^n$.
This proves the proposition.
\end{proof}

\section{The main results}
Let $k$ be a field, and $X$ be a separated of finite type $k$-scheme of dimension $d$. 
Let $M\ge 2$ be any integer.
Now we specialize the discussion in \Cref{Rec} to the situation when 
$$(X\times \A^1, M(X\times \{0\}))$$
is the modulus pair and introduce some notations that will be used throughout our proofs.
We will be using the dimensional notations (as opposed to codimensional notations) because our $X$ is not necessarily equidimensional.
\begin{para}
Let $Z\subset X\times \A^1\setminus\{0\}\times \square^q$ be a $q$-dimensional prime cycle. 
$W$ is the schematic image of $Z$ under the map 
$\bar \pr:X\times \P^1\times \bar \square^q\ra X$. 
$\bar Z$ is the closure of $Z$ in $X\times \P^1\times \bar \square^q$.
In the following diagram, the horizontal maps are normalizations, and the two unlabeled arrows are the natural open immersions. 
The maps $\pr: X\times \A^1\setminus\{0\}\times \square^q\ra X$, $\bar \pr: X\times \P^1\times \bar \square^q\ra X$ induce the maps $\pr:Z\ra W$ and $\bar \pr:\bar Z\ra W$ on the right. 
The maps $\tilde \pr:\tilde Z\ra \tilde W$ and $\tilde{\bar \pr}:\tilde{\bar Z}\ra \tilde W$ are given by the universal property of the normalization.
$$\xymatrix{
&
\tilde{\bar Z}
\ar[rr]^{\rho_{\bar Z}}\ar[ddl]^(.6){\tilde{\bar \pr}}
&&
\bar Z
\ar[ddl]^(.6){\bar \pr}
\\
\tilde Z
\ar[rr]^{\rho_Z} \ar@{^(->}[ur]\ar[d]_{\tilde \pr} 
&&
Z
\ar@{^(->}[ur]\ar[d]_{\pr}
\\
\tilde W
\ar[rr]^{\rho_W}
&&
W
}$$
We will use lowercase letters to represent dimension $q-1$ points of the corresponding scheme, e.g., $z\in Z_{(q-1)}, \tilde z\in \tilde Z_{(q-1)}, \tilde{\bar z}\in \tilde{\bar Z}_{(q-1)}, w\in W_{(q-1)}$, etc. ($w$ is not necessarily codimension 1 point in $W$!) The symbol $\eta$ with subscript will denote the generic point of the corresponding integral scheme, e.g., $\eta_W,\eta_Z$, etc.. $\rmv_{\tilde z}(-), \rmv_{\tilde{\bar z}}(-)$ etc., denotes the valuation.
\end{para} 
\begin{para} 
Define
\begin{align*}
\alpha:
\uline z_0(X\times \A^1| M(X\times \{0\}),q)&
\ra \Gamma(X,K_X^{-q})\\
Z&\mapsto
\begin{cases}
(-1)^q\Tr_{\bar Z\xra{\bar\pr}W\hra X}^{-q}(\frac{1}{t}d\log \{y_1,...,y_q\}),&
\text{$\bar Z\xra{\bar \pr} W$ is gen. fin.}\\
0,&
\text{$\bar Z\xra{\bar \pr} W$ is not gen. fin.}
\end{cases}
\end{align*}
Here $t$ is the coordinate for $\A^1$ and $y_1,\dots, y_q$ are the coordinates for $\square^q$ as introduced in \Cref{Rec},
and the rational form $\frac{1}{t}d\log \{y_1,...,y_q\}\in \Omega^q_{\bar Z,\eta_{\bar Z}}$ denotes the pullback of the obvious rational form via the closed immersion $\bar Z\hra X\times \P^1\times \bar \square^q$.
The notation $\Tr_{Y\xra{f} Y'}$ (with $f:Y\ra Y'$ being a proper morphism) denotes the trace map along $f$ in Serre-Grothendieck duality theory, see \cite[VI.\S4]{Ha66}\cite[\S3.4]{Co00}.

The map $\alpha$ maps the degenerate cycles to zero, because a prime cycle of the form $Z\times \square$ for some $Z\subset X\times (\A^1\setminus\{0\})\times \square^{q-1}$ cannot be generically finite over its image under the projection map $\bar\pr: X\times \P^1\times \bar \square^q\ra X$. 
\end{para} 
\bprop 
Let $M\ge 2$ be an integer.
The induced map 
$$\alpha:z_0(X\times \A^1| M(X\times \{0\}),\bullet) \ra 
\Gamma(X,K_X^{-\bullet})$$ 
is a chain map.
\eprop 
\bpf 
We have to show that the following diagram commutes
$$\xymatrix{
z_0(X\times \A^1| M(X\times \{0\}),q)
\ar[r]^(.7){\alpha}\ar[d]_{\partial^{\cyc}}&
\Gamma(X,K_X^{-q})
\ar[d]^{\partial^{\res}}
\\
z_0(X\times \A^1| M(X\times \{0\}),q-1)
\ar[r]^(.7){\alpha}&
\Gamma(X,K_X^{-q+1}).
}$$
Pick a prime cycle $Z\in z_0(X\times \A^1| M(X\times \{0\}),q)$. Let $W$ be the schematic image of $Z$ under the map $\bar \pr:X\times \P^1\times \bar \square^q\ra X$.
To show the commutativity of this diagram, it suffices to show
\beq\label{WTS alpha eq}
\partial^{\res}_{W,w}\circ \alpha'(Z)=
-\Tr^{-q+1}_{\bar{\{w\}}\hra W}\circ \alpha'_w\circ \partial^{\cyc}(Z)
\eeq 
in $\Gamma(W,K_W^{-q+1})$, for all $w\in W_{(q-1)}$. Here 
\beq\label{alphaprime}
\alpha'(Z):= \Tr_{\bar Z\xra{\bar\pr}W}^{-q}(\frac{1}{t}d\log \{y_1,...,y_q\})
\in \Gamma(W,K_W^{-q}).
\eeq
Note that this definition works no matter whether $\bar\pr:\bar Z\ra W$ is generically finite or not: when it is not generically finite, $\alpha'(Z)=0$ because $K_W^{-q}=0$.
For any $z$ being the generic point of an irreducible component of some face of $Z$, 
$$\alpha'_w(\bar{\{z\}}^Z):=
\begin{cases}
\Tr^{-q+1}_{\bar{\{z\}}^{\bar Z}\xra{\bar \pr} \bar{\{w\}}^{W}}(\frac{1}{t}d\log \{y_1,...,y_q\}),
&\text{when $\dim \bar{\{z\}}^Z =q-1$ and $\pr(z)=w$,}
\\
0,&\text{else.}
\end{cases}
$$
It's easy to see that
\benu 
\item 
$\bar\pr: \bar Z\ra W$ being generically finite, and
\item 
$\bar\pr: \bar Z\ra W$ being of generic dimension 1
\eenu 
are the only two nontrivial cases. 
No matter whether $\bar\pr: \bar Z\ra W$ is generically finite or of generic dimension 1, the restricted morphism
$\bar{\{z\}}^{\bar Z} \xra{\bar\pr} \bar{\{w\}}^{W}$
is generically finite, according to the general dimension formula \cite[Proposition 5.6.5]{EGAIV-2}.
(Formally, we can say that $\alpha'$ is a map from the subgroup of $z_0(X\times \A^1| M(X\times \{0\}),q)$ generated by $Z$ to $\Gamma(W,K_W^{-q})$, and 
$\alpha'_w$ is a map from the subgroup of $z_0(X\times \A^1| M(X\times \{0\}),q-1)$ generated by the irreducible components (but equipped with reduced scheme structure) of faces of $Z$ to $\Gamma(\bar{\{w\}}^{W},K_{\bar{\{w\}}^{W}}^{-q+1})$. But we won't need this.
Apparently these $\alpha'$ and $\alpha'_w$ depend on $Z$ and are just extracted from the definition of $\alpha$ to facilitate our computation later on.)

Before diving into the discussion in cases, we do some simplifications that work for all cases.

\begin{align}
\label{partial res circ alpha}
\partial^{\res}_{W,w}\circ \alpha'(Z)
&=
\partial^{\res}_{W,w}\left(
\Tr^{-q}_{\bar Z\xra{\bar \pr}W}(\frac{1}{t}d\log \{y_1,...,y_q\})
\right)
\\
&=
\partial^{\res}_{W,w}\left(
\Tr^{-q}_{\tilde{\bar Z}\xra{\rho_{\bar Z}}\bar Z\xra{\bar \pr}W}
(\frac{1}{t}d\log \{y_1,...,y_q\})
\right)
\nonumber
\\
&=
\sum_{\tilde{\bar z}\in (\bar \pr \circ \rho_{\bar Z})^{-1}(w)}
\Tr^{-q+1}_{\tilde{\bar Z}\xra{\rho_{\bar Z}}\bar Z\xra{\bar \pr}W}
\left(
\partial^{\res}_{\tilde{\bar Z},\tilde{\bar z}}
(\frac{1}{t}d\log \{y_1,...,y_q\})
\right).
\nonumber
\end{align}
The first equality is just the definition of $\alpha'$.
The second equality holds because the trace map $\Tr_{\tilde{\bar Z}\xra{\rho_{\bar Z}}\bar Z}$ is the identity map at degree $-q$. The third equality holds because the trace map $\Tr_{\tilde{\bar Z}\xra{\rho_{\bar Z}}\bar Z\xra{\bar \pr}W}: K_{\tilde{\bar Z}}\ra K_W$ is a chain map.

Before the next calculation, observe that for any $z\in Z^{(1)}$, it can lie in at most one face. This is because $Z$ intersects all faces properly: if $z$ lies in two faces, its codimension is at least 2.
So when $z\in Z^{(1)}$ lies in some face, it makes sense to write $z\in \{y_{i(z)}=\epsilon(z)\}$ for some $i(z)\in \{1,...,q\}$ and $\epsilon(z)\in\{0,\infty\}$. 
When we write $i(\tilde z)$ for some $\tilde z\in \tilde Z$, we mean $i(\rho_Z(\tilde z))$. Similarly the symbol $\epsilon(\tilde z)$ might also occur.
When $z\in Z^{(1)}$ does not lie in any face, define the symbols $\ord_z(y_{i(z)})$ and $\rmv_{\tilde z}(y_{i(\tilde z)})$ to be zero. 
Now calculate $\alpha'_w\circ \partial^{\cyc}(Z)$:
\begin{align}
\label{alpha circ partial cyc}
\alpha'_w\circ \partial^{\cyc}(Z)
&=
\alpha'_w\left(
\sum_{i=1}^q (-1)^i
[(\id_X\times \iota_{i,\infty}^q)^*(Z)-(\id_X\times \iota_{i,0}^q)^*(Z)]
\right)
\\
&=
\alpha'_w\left(
\sum_{i=1}^q (-1)^{i-1}
\div(y_i)
\right)
\nonumber
\\
&=
\alpha'_w\left(
\sum_{i=1}^q (-1)^{i-1}
\left[
\sum_{z\in Z^{(1)}}
\ord_z(y_{i})\cdot \bar{\{z\}}^Z
\right]
\right)
\nonumber
\\
&=
\sum_{z\in \pr^{-1}(w)\cap Z^{(1)}}
(-1)^{i(z)-1} \ord_z(y_{i(z)}) 
\Tr^{-q+1}_{\bar{\{z\}}^{\bar Z}\xra{\bar \pr} \bar{\{w\}}^W}
(\frac{1}{t}d\log \{y_1,...,\hat{y_{i(z)}},...,y_q\})
\nonumber
\\
&=
\sum_{z\in \pr^{-1}(w)\cap Z^{(1)}}
(-1)^{i(z)-1} \left(
\sum_{\tilde{z}\in \rho_{Z}^{-1}(z)}
\rmv_{\tilde z}(y_{i(z)})[k(\tilde z):k(z)]
\right)
\nonumber
\\
&\qquad\qquad\qquad\qquad\qquad
\cdot
\Tr^{-q+1}_{\bar{\{z\}}^{\bar Z}\xra{\bar \pr} \bar{\{w\}}^W}
(\frac{1}{t}d\log \{y_1,...,\hat{y_{i(z)}},...,y_q\})
\nonumber
\\
&=
\sum_{z\in \pr^{-1}(w)\cap Z^{(1)}}
(-1)^{i(z)-1}
\sum_{\tilde{z}\in \rho_{Z}^{-1}(z)}
\rmv_{\tilde z}(y_{i(z)})
\Tr^{-q+1}_{\bar{\{\tilde z\}}^{\tilde{\bar Z}}\xra{\bar \pr\circ \rho_{\bar Z}} \bar{\{w\}}^W}
(\frac{1}{t}d\log \{y_1,...,\hat{y_{i(z)}},...,y_q\})
\nonumber
\\
&=
\sum_{\tilde{z}\in (\pr\circ \rho_{Z})^{-1}(w)\cap \tilde Z^{(1)}}
(-1)^{i(\tilde z)-1}
\rmv_{\tilde z}(y_{i(\tilde z)})
\Tr^{-q+1}_{\bar{\{\tilde z\}}^{\tilde{\bar Z}}\xra{\bar \pr\circ \rho_{\bar Z}} \bar{\{w\}}^W}
(\frac{1}{t}d\log \{y_1,...,\hat{y_{i(\tilde z)}},...,y_q\}).
\nonumber
\end{align}
The first four equalities hold by definition of the symbols involved. 
The 5th equality holds by \cite[Example 1.2.3]{Fulton-Intersection}. 
The 6th one holds because the trace map $\Tr_{\bar{\{\tilde z\}}^{\tilde{\bar Z}} \xra{\rho_{\bar Z}} \bar{\{z\}}^{\bar Z}}$ of the finite morphism $\bar{\{\tilde z\}}^{\tilde{\bar Z}} \xra{\rho_{\bar Z}} \bar{\{z\}}^{\bar Z}$ of dimension $(q-1)$-schemes at degree $-q+1$ is given by the classical trace map $\Omega^{q-1}_{k(\tilde z)}\ra \Omega^{q-1}_{k(z)}$. 
Since the element $\frac{1}{t}d\log \{y_1,...,\hat{y_{i(z)}},...,y_q\} \in \Omega^{q-1}_{k(\tilde z)}$ comes from the pullback of the differential form $\frac{1}{t}d\log \{y_1,...,\hat{y_{i(z)}},...,y_q\} \in \Omega^{q-1}_{k(z)}$, the trace of it is given by multiplication by the degree of the field extension $k(z)\subset k(\tilde z)$.
The 7th equality is just a recollection.

We will use the following decomposition of the underlying set of $\bar\pr^{-1}(w)$:
\begin{align*}
\bar\pr^{-1}(w)=&
\underbrace{\left(\bar\pr^{-1}(w)\cap Z\right)}_{a)}\cup
\underbrace{\left(\bar\pr^{-1}(w)\cap \{t=0\}\right)}_{b)}\cup
\\
&\qquad\qquad\qquad
\underbrace{(\bar\pr^{-1}(w)\cap \big(\bigcup_{i=1}^q \{y_i=1\}\setminus \{t=0\text{ or }\infty\}\big) )}_{c)}\cup
\underbrace{\left(\bar\pr^{-1}(w)\cap\{t=\infty\}\right)}_{d)}
.
\end{align*}

\begin{enumerate}[label=\alph*)]
\item 
\textit{Good position.}
Let $z\in \bar\pr^{-1}(w)\cap Z$. Then if $z$ lies in some face, it lies in at most one face $\{y_{i(z)}=\epsilon(z)\}$ as we discussed above. 
For $\tilde{\bar z}\in \tilde{\bar Z}$, denote by $\pi_{\tilde{\bar z}}$ the uniformizer of the discrete valuation ring $\cO_{\tilde{\bar Z},\tilde{\bar z}}$.
For a fixed $i$, one has $y_i=u(i,\tilde{\bar z})\cdot \pi_{\tilde{\bar z}}^{\rmv_{\tilde{\bar z}}(y_i)}$ for some $u(i,\tilde{\bar z})\in \cO_{\tilde{\bar Z},\tilde{\bar z}}^*$.


Continuing the calculation of  (\ref{partial res circ alpha}) over the points $\tilde{\bar z}\in (\bar \pr \circ \rho_{\bar Z})^{-1}(w)\cap \tilde Z$:
\begin{align*}
&\quad
\sum_{\tilde{\bar z}\in (\bar \pr \circ \rho_{\bar Z})^{-1}(w)\cap \tilde Z^{(1)}}
\Tr^{-q+1}_{\tilde{\bar Z}\xra{\rho_{\bar Z}}\bar Z\xra{\bar \pr}W}
\left(
\partial^{\res}_{\tilde{\bar Z},\tilde{\bar z}}
(\frac{1}{t}d\log \{y_1,...,y_q\})
\right)
\\
&=
\sum_{\tilde{\bar z}\in (\bar \pr \circ \rho_{\bar Z})^{-1}(w)\cap \tilde Z^{(1)}}
(-1)^{i(\tilde{\bar z})-1}
\Tr^{-q+1}_{\tilde{\bar Z}\xra{\rho_{\bar Z}}\bar Z\xra{\bar \pr}W}
\left(
\partial^{\res}_{\tilde{\bar Z},\tilde{\bar z}}
(\frac{1}{t}d\log \{y_{i(\tilde{\bar z})},y_1,...,\hat{y_{i(\tilde{\bar z})}},...,y_q\})
\right)
\\
&=
\sum_{\tilde{\bar z}\in (\bar \pr \circ \rho_{\bar Z})^{-1}(w)\cap \tilde Z^{(1)}}
(-1)^{i(\tilde{\bar z})-1}
\Tr^{-q+1}_{\tilde{\bar Z}\xra{\rho_{\bar Z}}\bar Z\xra{\bar \pr}W}
\left(
\partial^{\res}_{\tilde{\bar Z},\tilde{\bar z}}
(\frac{1}{t}
d\log \{\pi_{\tilde{\bar z}}^{\rmv_{\tilde{\bar z}}(y_{i(\tilde{\bar z})})}
,y_1,...,\hat{y_{i(\tilde{\bar z})}},...,y_q\})
\right)
\\
&=
\sum_{\tilde{\bar z}\in (\bar \pr \circ \rho_{\bar Z})^{-1}(w)\cap \tilde Z^{(1)}}
(-1)^{i(\tilde{\bar z})-1}
\rmv_{\tilde{\bar z}}(y_{i(\tilde{\bar z})})
\Tr^{-q+1}_{\tilde{\bar Z}\xra{\rho_{\bar Z}}\bar Z\xra{\bar \pr}W}
\left(
\left[
  \begin{array}{c}
   \frac{1}{t} d\pi_{\tilde{\bar z}}\wedge
    d\log \{y_1,...,\hat{y_{i(\tilde{\bar z})}},...,y_q\}
   \\
   \pi_{\tilde{\bar z}}\\
  \end{array}
\right]
\right)
\\
&=
\sum_{\tilde{\bar z}\in (\bar \pr \circ \rho_{\bar Z})^{-1}(w)\cap \tilde Z^{(1)}}
(-1)^{i(\tilde{\bar z})}
\rmv_{\tilde{\bar z}}(y_{i(\tilde{\bar z})})
\Tr^{-q+1}_{\tilde{\bar Z}\xra{\rho_{\bar Z}}\bar Z\xra{\bar \pr}W}
\Tr^{-q+1}_{\bar{\{\tilde z\}}^{\tilde{\bar Z}}\hra \tilde{\bar Z}}
(\frac{1}{t}d\log \{y_1,...,\hat{y_{i(\tilde{\bar z})}},...,y_q\})
\\
&=
\sum_{\tilde{\bar z}\in (\bar \pr \circ \rho_{\bar Z})^{-1}(w)\cap \tilde Z^{(1)}}
(-1)^{i(\tilde z)}
\rmv_{\tilde z}(y_{i(\tilde{\bar z})})
\Tr^{-q+1}_{\bar{\{w\}}^W\hra W}
\Tr^{-q+1}_{\bar{\{\tilde z\}}^{\tilde{\bar Z}}\xra{\bar \pr\circ \rho_{\bar Z}} \bar{\{w\}}^W}
(\frac{1}{t}d\log \{y_1,...,\hat{y_{i(\tilde{\bar z})}},...,y_q\}).
\end{align*}
The 1st equality is just a reordering and thus produces a sign.
The 2nd equality holds because $\frac{1}{t}
d\log \{u(i(\tilde{\bar z}),\tilde{\bar z})
,y_1,...,\hat{y_{i(\tilde{\bar z})}},...,y_q\}$ is regular at point $\tilde{\bar z}$ and thus 
$\partial^{\res}_{\tilde{\bar Z},\tilde{\bar z}}
(\frac{1}{t}
d\log \{u(i(\tilde{\bar z}),\tilde{\bar z})
,y_1,...,\break\hat{y_{i(\tilde{\bar z})}},...,y_q\})=0$.
The 3rd equality is given by \cite[Lemma A.1.2]{CR11}.
The 4th equality is given by \cite[Lemma A.2.12]{CR11}.
The last equality is just the compatibility of the trace map with composition of proper morphisms. Compare this last line to (\ref{alpha circ partial cyc}), we know that we have arrived at the equality (\ref{WTS alpha eq}) above the points in $\rho_{\bar Z}^{-1}(\bar\pr^{-1}(w)\cap Z)$.


\item 
\textit{Modulus condition.}
When $\tilde{\bar z}\in\rho_{\bar Z}^{-1}(\bar\pr^{-1}(w)\cap \{t=0\})$, 
$\alpha'_w\circ \partial^{\cyc}(Z)$ has no summand supported on this point. We thus have to show that the summand of $\partial^{\res}_{W,w}\circ \alpha'(Z)$ supported on point $\tilde{\bar z}$ vanishes. 
For this one needs the modulus condition \eqref{defn:modulus-cycle1}: since $\tilde{\bar z}$ is a zero of $t$, it must lie in some $\{y_i=1\}$ for at least one $i\in \{1,...,q\}$. 
Reordering the indices (which might create a sign but does not affect the vanishing), we assume that $\tilde{\bar z}\in\{y_i=1\}$ for $i\in\{1,,...,r\}$ for some $1\le r\le q$, and $\tilde{\bar z}\notin\{y_i=1\}$ for $i\in\{r+1,...,q\}$. 
So we write
$$y_i-1=u_i\cdot \pi_{\tilde{\bar z}}^{m_i}, \quad
u_i\in \cO_{\tilde{\bar Z},\tilde{\bar z}}^*,\quad
m_i\ge 1,\quad i\in\{1,,...,r\},
$$
$$y_i=u_i\cdot \pi_{\tilde{\bar z}}^{m_i}, \quad
u_i\in \cO_{\tilde{\bar Z},\tilde{\bar z}}^*,\quad
m_i\ge 0,\quad i\in\{r+1,,...,q\},$$
$$t=u\cdot \pi_{\tilde{\bar z}}^{m},\quad
u_i\in \cO_{\tilde{\bar Z},\tilde{\bar z}}^*,\quad
m\ge 1.$$
In $\Omega^q_{\tilde{\bar Z},\eta_{\tilde{\bar Z}}}$,
\begin{align*} 
\frac{1}{t}d\log \{y_1,...,y_q\}=
\frac{1}{u \pi_{\tilde{\bar z}}^{m}}
\frac{d(1+u_1 \pi_{\tilde{\bar z}}^{m_1})}{1+u_1 \pi_{\tilde{\bar z}}^{m_1}}
\wedge\dots\wedge
\frac{d(1+u_r \pi_{\tilde{\bar z}}^{m_r})}{1+u_r \pi_{\tilde{\bar z}}^{m_r}}
\wedge
\frac{d(u_{r+1} \pi_{\tilde{\bar z}}^{m_{r+1}})}{u_{r+1} \pi_{\tilde{\bar z}}^{m_{r+1}}}
\wedge\dots\wedge
\frac{d(u_q \pi_{\tilde{\bar z}}^{m_q})}{u_q \pi_{\tilde{\bar z}}^{m_q}}
\end{align*}
Since $\frac{1}{u}\frac{1}{1+u_1 \pi_{\tilde{\bar z}}^{m_1}}
\wedge\dots\wedge
\frac{1}{1+u_r \pi_{\tilde{\bar z}}^{m_r}}$
is regular at point $\tilde{\bar z}$, we are only concerned with the part
\eq{eqmoduluscase}{\frac{1}{\pi_{\tilde{\bar z}}^{m}}
d(1+u_1 \pi_{\tilde{\bar z}}^{m_1})\wedge\dots\wedge
d(1+u_r \pi_{\tilde{\bar z}}^{m_r})\wedge
\frac{d(u_{r+1} \pi_{\tilde{\bar z}}^{m_{r+1}})}{u_{r+1} \pi_{\tilde{\bar z}}^{m_{r+1}}}
\wedge\dots\wedge
\frac{d(u_q \pi_{\tilde{\bar z}}^{m_q})}{u_q \pi_{\tilde{\bar z}}^{m_q}}.
}
Since $d(u_i \pi_{\tilde{\bar z}}^{m_i})= 
\pi_{\tilde{\bar z}}^{m_i}du_i+ m_iu_i\pi_{\tilde{\bar z}}^{m_i-1}d\pi_{\tilde{\bar z}}$ and any two terms with $d\pi_{\tilde{\bar z}}$ in the wedge product cancel out, we deduce that \eqref{eqmoduluscase} is the product of 
$\pi_{\tilde{\bar z}}^{(m_1+...+m_r)-1-m}$
with a regular differential form. 
The modulus condition \eqref{defn:modulus-cycle1} states
$$Mm\le m_1+...+m_r.$$
As $m+1\le Mm$, $\frac{1}{t}d\log \{y_1,...,y_q\}$ has no pole at point $\tilde{\bar z}$. 
In other words, $\frac{1}{t}d\log \{y_1,...,y_q\}\in \Omega^q_{\tilde{\bar Z},\tilde{\bar z}}$, so $\partial^{\res}_{\tilde{\bar Z},\tilde{\bar z}}(\frac{1}{t}d\log \{y_1,...,y_q\})=0$.

\item 
\textit{Zero symbols.}
When $\tilde{\bar z}\in\rho_{\bar Z}^{-1}(\bar\pr^{-1}(w)\cap \big(\bigcup_{i=1}^q \{y_i=1\}\setminus \{t=0\text{ or }\infty\}\big) )$, 
$\alpha'_w\circ \partial^{\cyc}(Z)$ have no summand supported on this point. So the aim is still to show that the summand of $\partial^{\res}_{W,w}\circ \alpha'(Z)$ supported on point $\tilde{\bar z}$ vanishes. This is true because the symbol $\{y_1,...,y_q\}\in K^M_q(k(\eta_{\tilde{\bar Z}}))$ is zero when one of the $y_i$ is 1. (It's actually already zero in the tensor product $k(\eta_{\tilde{\bar Z}})^*\otimes\cdots\otimes k(\eta_{\tilde{\bar Z}})^*$ before modding out the Steinberg relations.)

\item 
\textit{A priori regular locus.}
When $\tilde{\bar z}\in\rho_{\bar Z}^{-1}(\bar\pr^{-1}(w)\cap \{t=\infty\})$, $\alpha'_w\circ \partial^{\cyc}(Z)$ still have no summand supported on this point. We claim that 
$\frac{1}{t}d\log \{y_1,...,y_q\}\in \Omega^q_{\tilde{\bar Z},\tilde{\bar z}}$. This is because of the property of log poles: the rational form $d\log \{y_1,...,y_q\}$ has at most a pole of order 1 at point $\tilde{\bar z}$. And $\frac{1}{t}$ has at least a zero of order 1, thus the whole form $\frac{1}{t}d\log \{y_1,...,y_q\}$ is regular at point $\tilde{\bar z}$.
\end{enumerate}

To summarize: 
\begin{align}
\label{summerize alpha eq}
-\Tr^{-q+1}_{\bar{\{w\}}^W\hra W}\circ\alpha'_{w}&\circ \partial^{\cyc}(Z)
\\
\nonumber
&=
\sum_{\tilde{\bar z}\in (\bar \pr\circ \rho_{\bar Z})^{-1}(w)\cap \tilde Z}
(-1)^{i(\tilde{\bar z})}
\rmv_{\tilde{\bar z}}(y_{i(\tilde{\bar z})})
\Tr^{-q+1}_{\bar{\{\tilde{\bar z}\}}^{\tilde{\bar Z}}\xra{\bar \pr\circ \rho_{\bar Z}} W}
(\frac{1}{t}d\log \{y_1,...,\hat{y_{i(\tilde{\bar z})}},...,y_q\})
\\
\nonumber
&=
\sum_{\tilde{\bar z}\in (\bar \pr\circ \rho_{\bar Z})^{-1}(w)\cap \tilde Z}
\Tr^{-q+1}_{\tilde{\bar Z}\xra{\bar \pr\circ \rho_{\bar Z}} W}
\left(\partial^{\res}_{\tilde{\bar Z},\tilde{\bar z}}(\frac{1}{t}d\log \{y_1,...,y_{i(\tilde{\bar z})},...,y_q\})\right)
\text{\quad (by a))}
\nonumber
\\
&=
\sum_{\tilde{\bar z}\in (\bar \pr\circ \rho_{\bar Z})^{-1}(w)\cap \tilde Z}+
\sum_{\tilde{\bar z}\in (\bar \pr\circ \rho_{\bar Z})^{-1}(w)\cap \{t=0\}}+
\nonumber
\\
&\qquad
\sum_{\tilde{\bar z}\in (\bar \pr\circ \rho_{\bar Z})^{-1}(w)\cap
\left(
\bigcup_{i=1}^q \{y_i=1\}\setminus\{t=0\text{ or }\infty\}
\right)
}+
\sum_{\tilde{\bar z}\in (\bar \pr\circ \rho_{\bar Z})^{-1}(w)\cap \{t=\infty\}}
\text{\quad (by b) c) d))}
\nonumber
\\
&=
\sum_{\tilde{\bar z}\in (\bar \pr\circ \rho_{\bar Z})^{-1}(w)}
\Tr^{-q+1}_{\tilde{\bar Z}\xra{\bar \pr\circ \rho_{\bar Z}} W}
\left(\partial^{\res}_{\tilde{\bar Z},\tilde{\bar z}}(\frac{1}{t}d\log \{y_1,...,y_{i(\tilde{\bar z})},...,y_q\})\right)
\nonumber
\\
&=
\partial^{\res}_{W,w}\circ \alpha'(Z).
\nonumber
\end{align}
Thus we establish (\ref{WTS alpha eq}).

\epf 

\brmk 
\begin{enumerate} 
\item
\label{rmk on trace vs reciprocity}
Compare the proof presented here and the proof of \cite[Lemma 3.1, 2nd case]{RS18}, one might wonder why the reciprocity property doesn't appear here. In fact, the last equality in (\ref{partial res circ alpha}) is a reciprocity phenomenon.
To see this, 
consider the case when $\bar \pr:\bar Z\ra W$ is not generically finite and $\dim W=q-1$. Take $w=\eta_W$ in this case. Then $(\bar \pr \circ \rho_{\bar Z})^{-1}(\eta_W)$ is a regular proper curve over $\eta_W$ (thus also projective), and 
$\partial^{\res}_{W,w}\circ \alpha'(Z)=0$ 
 since $K_W^{-q}=0$. 
Hence (\ref{partial res circ alpha}) gives
$$\sum_{\tilde{\bar z}\in (\bar \pr \circ \rho_{\bar Z})^{-1}(\eta_W)}
\Tr^{-q+1}_{\tilde{\bar Z}\xra{\rho_{\bar Z}}\bar Z\xra{\bar \pr}W}
\left(
\partial^{\res}_{\tilde{\bar Z},\tilde{\bar z}}
(\frac{1}{t}d\log \{y_1,...,y_q\})
\right)=0.$$
This is precisely the reciprocity statement appearing in \textit{loc. cit.}.
\item 
Note that the map $\alpha$ factors through the cycle complex with modulus with $M=2$.
\end{enumerate} 
\ermk

\begin{theorem}
\label{thm}
Let $M\ge 2$ be an integer.
The map 
$$\alpha:\Z(d+1)_{X\times \A^1 \mid M\cdot(X\times \{0\})}[2d+2]\ra K_X$$
is a chain map between complexes of \'etale sheaves.
\end{theorem}
\begin{proof}
Let $g:U\ra X$ be an \'etale morphism. 
We want to show that the following diagram commutes:
\eq{etalecomm}{
\xymatrix{
z_0(X\times \A^1| M(X\times \{0\}),q)
\ar[r]^(.7){\alpha}\ar[d]_{(g\times \id)^*}&
\Gamma(X,K_X^{-q})
\ar[d]^{g^*}
\\
z_0(U\times \A^1| M(U\times \{0\}),q)
\ar[r]^(.7){\alpha}&
\Gamma(U,K_X^{-q})
.}
}
Pick a prime cycle $Z\in z_0(X\times \A^1| M(X\times \{0\}),q)$. 
Let $W$ be the schematic image of $Z$ under the map $\bar \pr:X\times \P^1\times \bar \square^q\ra X$.
Let $\eta_1,\dots, \eta_n$ be the generic points of the $q$-dimensional irreducible components of $W_U:=W\times_X U$. 
For each $i=1,\dots,n$, let $\eta_{i,1},\dots,\eta_{i,m_i}$ be the
generic points of the $q$-dimensional irreducible components of 
$\bar Z_U:=\bar Z\times_X U$ which map to $\eta_i$ under $\bar\pr_U$.
We have a cartesian diagram
$$\xymatrix{
*+[l]{
U\times \P^1\times \bar \square^q 
\supset \bar Z_U}
\ar[r]^{(g\times \id)^*}\ar[d]_{\bar \pr_U}
&
*+[r]{\bar Z\subset X\times \P^1\times \bar \square^q}
\ar[d]^{\bar \pr}
\\
*+[l]{U\supset W_U}
\ar[r]^g
&*+[r]{W\subset X.}
}$$
To show the commutativity of \eqref{etalecomm}, it suffices to show
\beq\label{307}
g^*\circ \alpha'(Z)=
\sum_{i=1}^n \sum_{j=1}^{m_i} 
\alpha'_{\eta_{i,j}}\circ (g\times \id)^*(Z)
\eeq 
in $\Gamma(W_U,K_{W_U}^{-q})$.
Here $\alpha'$ is defined in \eqref{alphaprime}, and
$$\alpha_{\eta_{i,j}}'(\bar{\{\eta_{i,j}\}}):= 
\Tr_{\bar{\{\eta_{i,j}\}}\xra{\bar\pr_U}\bar{\{\eta_{i}\}}}^{-q}(\frac{1}{t}d\log \{y_1,...,y_q\})
\in \Gamma(W_U,K_{W_U}^{-q}). 
$$
When $\bar \pr:\bar Z\ra W$ is not generically finite, $\Gamma(W_U,K_{W_U}^{-q})=0$ and hence we have the commutativity.
When $\bar \pr:\bar Z\ra W$ is generically finite, 
$\Gamma(W_U,K_{W_U}^{-q})=\bigoplus_{i=1}^n \Omega^q_{k(\eta_i)/k}$, and
\begin{align*}
g^*\circ \alpha'(Z)
&=
g^*\left(\Tr_{\bar Z\xra{\bar\pr}W}^{-q}(\frac{1}{t}d\log \{y_1,...,y_q\})\right)
\stackrel{(*)}{=}\Tr_{\bar Z_U\xra{\bar\pr_U}W_U}^{-q}(\frac{1}{t}
d\log \{y_1,...,y_q\})
\\
&=\sum_{i=1}^n \sum_{j=1}^{m_i} 
\Tr_{\bar{\{\eta_{i,j}\}}\xra{\bar\pr_U}\bar{\{\eta_{i}\}}}^{-q}(\frac{1}{t}d\log \{y_1,...,y_q\})
=
\sum_{i=1}^n \sum_{j=1}^{m_i} 
\alpha'_{\eta_{i,j}}\circ (g\times \id)^*(Z)
\end{align*}
The equality $(*)$ follows from the compatibility of the trace map of residual complexes and \'etale pullbacks.
The remaining equalities hold by definition.
Therefore \eqref{307} follows.
\end{proof}

\begin{para}  
\label{phiprime}
Recall the map $\zeta$ from \cite[\S5.1]{RenThesis}
and $\phi':=\phi'(0)$ from \Cref{phiprimer}. 
Since the norm map from Milnor $K$-theory is compatible with the trace map from coherent duality theory \cite[Lemma 5.3]{RenThesis}, we have the following explicit formula for the composite map $\zeta\circ\varphi':
\Z^{\rm cube}_X:=\Z(d)_{X|0}[2d]\ra K_{X,log}:=\Cone(C-1:K_X\ra K_X)[-1]$: 
\begin{align*}
\zeta\circ \varphi': 
\uline z_0(X,q)
&\ra 
\Gamma(X,K_X^{-q})\oplus\Gamma(X,K_X^{-q-1})
\\
Z&\mapsto
\begin{cases}
\left(
(-1)^q\Tr_{\bar Z\xra{\bar\pr}W\hra X}^{-q}(d\log \{y_1,...,y_q\}),0
\right),&
\text{$\bar Z\xra{\bar \pr} W$ is gen. fin.}\\
(0,0),&
\text{$\bar Z\xra{\bar \pr} W$ is not gen. fin.}
\end{cases}
\end{align*}
Here $Z$ is a prime cycle of dimension $q$, and $W$ is the schematic image of $\bar Z$ under the map $\bar\pr:X\times \P^1\times\bar \square^q\ra X$.
\end{para}
\begin{prop}
\label{Compatibility}
Suppose $k$ is a perfect field of positive characteristic $p$. 
\footnote{The perfectness of $k$ is only needed to have the Cartier operator $C:K_X\ra K_X$ so that $K_{X,\log}$ can be defined. See \cite[\S1.2]{RenThesis}.}
The map $\alpha$ is compatible with $\zeta\circ \varphi':\Z^{\rm cube}_X\ra K_{X,log}$.
Namely,
\eq{CompDiag}{\xymatrix{
\Z^{\rm cube}_X
\ar[r]^(0.27){i_{1*}}
\ar[d]_{\zeta\circ \varphi'}&
\Z(d+1)_{X\times \A^1 \mid M(X\times \{0\})}[2d+2]
\ar[d]^{\alpha}
\\
K_{X,log}
\ar[r]^{(a,b)\mapsto a}
&K_X.
}}
is a commutative diagram of \'etale sheaves.
\end{prop}
\begin{proof}
We need to show that the following diagram commutes:
$$\xymatrix{
\uline z_0(X,q)
\ar[r]^(0.4){i_{1*}}\ar[d]_{\zeta\circ \varphi'}
&
\uline z_0(X\times \A^1| M(X\times \{0\}),q)
\ar[d]^{\alpha}
\\
\Gamma(X,K_X^{-q})\oplus\Gamma(X,K_X^{-q-1})
\ar[r]^(0.6){(a,b)\mapsto a}
&
\Gamma(X,K_X^{-q}).
}$$
This follows directly from the construction of the maps involved.
\end{proof}

\begin{corollary}
\label{injectivity}
Let $k$ be an algebraically closed field of positive characteristic $p$.
Then the following diagram commutes, and the two horizontal maps in the diagram are injections
$$\xymatrix{
\CH_0(X,q; \Z/p) \ar[r]^(0.33){i_{1*}}\ar@{=}[d] &
\CH_0(X\times \A^1| M(X\times {0}), q; \Z/p)\ar[d]
\\
H^{2d-q}_{\cM,\et}(X_\et, \Z/p(d)) \ar[r]^(0.33){i_{1*}} &
H^{2d+2-q}_{\cM,\et}({X\times \A^1| M(X\times {0})}, \Z/p(d+1)).
}$$
\end{corollary}
\begin{proof}
Consider the following commutative diagram
(we omit the subscript ``${X\times \A^1 | M(X\times \{0\})}$'' from the right face due to formatting reasons)
$$\resizebox{\displaywidth}{!}{
\xymatrix{
&
R\Gamma(X_\et,\Z/p(d)_{X}[2d])
\ar[rr]^{i_{1*}}
\ar[dd]_(0.65){\zeta\circ \varphi'}
&&
R\Gamma(X_\et,\Z/p(d+1)[2d+2])
\ar[dd]^{\alpha}
\\
\Gamma(X,\Z/p(d)_{X}[2d])
\ar[ru]^{}
\ar[rr]^(0.27){i_{1*}}
\ar[dd]_{\zeta\circ \varphi'}
&&
\Gamma(X,\Z/p(d+1)[2d+2])
\ar[ur]
\ar[dd]^(0.35){\alpha}
\\
&
R\Gamma(X,K_{X,log})
\ar[rr]^(0.35){(a,b)\mapsto a}
&&R\Gamma(X,K_X)
\\
\Gamma(X,K_{X,log})
\ar[ur]^{\simeq}
\ar[rr]^{(a,b)\mapsto a}
&&\Gamma(X,K_X)
\ar[ur]^{\simeq}.
}}$$
The two ``$\simeq$'' in the diagram follow from the fact that both $K_{X,\log}$ and $K_X$ are complexes of injectives.
Taking the  $-q$-th cohomology of the whole diagram, 
we get
$$\resizebox{\displaywidth}{!}{
\xymatrix{
&
H^{2d-q}_{\cM,\et}(X,\Z/p(d))
\ar[rr]^{{i_{1*}},\circled{g}}
\ar[dd]_(0.65){\zeta\circ\phi',\circled{e}}
&&
H^{2d+2-q}_{\cM,\et}(X\times \A^1 | M(X\times \{0\}),\Z/p(d+1))
\ar[dd]^{\alpha}
\\
\CH_{0}(X,q;\Z/p)
\ar[ru]^{\simeq,\circled{a}}
\ar[rr]^(0.27){i_{1*},\circled{d}}
\ar[dd]_{\zeta\circ \varphi',\circled{b}}
&&
\CH_{0}(X\times \A^1 | M(X\times \{0\}),q;\Z/p)
\ar[ur]
\ar[dd]^(0.35){\alpha}
\\
&
H^{2d-q}(X,K_{X,log})
\ar[rr]^(0.3){(a,b)\mapsto a,\circled{f}}
&&
H^{2d-q}(X,K_X)
\\
H^{2d-q}(X,K_{X,log})
\ar[ur]^{\simeq}
\ar[rr]^{(a,b)\mapsto a,\circled{c}}
&&
H^{2d-q}(X,K_X)
\ar[ur]^{\simeq}.
}}$$
In this diagram,
\begin{itemize}
    \item 
    $\circled{a}$ is an isomorphism due to the \'etale descent for higher Chow groups when $k=\bar k$, see \cite[Proposition 8.8]{RenThesis}.
    \item 
    $\circled{b}, \circled{e}$ are isomorphisms by \cite[Theorem 6.1]{RenThesis}.
    \item 
    \circled{c} is an injection by \cite[Proposition 8.2]{RenThesis}.
    \item 
    The commutativity of the front face hence implies that $\circled{d}$ is injective.
    \item 
     The commutativity of the back face hence implies that $\circled{g}$ is injective (since $\circled{f}$ is injective).
\end{itemize}
\end{proof}

\bigskip
\textbf{Acknowledgments}.
The author thanks Kay R\"ulling for raising the question in the introduction and for related discussions.

\medskip
\printbibliography 

\bigskip 


\end{document}